\documentclass{article}
\usepackage[utf8]{inputenc}  

\usepackage[english]{babel}
\usepackage{csquotes}

\usepackage{biblatex}
\usepackage{arxiv}

\usepackage[T1]{fontenc}    
\usepackage{url}                 
\usepackage{booktabs}        
\usepackage{amsfonts}        
\usepackage{nicefrac}          
\usepackage{microtype}       

\usepackage{amsthm} 
\usepackage{amsmath}

\usepackage{pgfplots}
\pgfplotsset{width=8cm,compat=1.9}
\usepgfplotslibrary{fillbetween}
\usepackage{tikz}
\usetikzlibrary{backgrounds}

\newcommand{\R}{\mathbb{R}}
\newcommand{\adj}{\operatorname{adj}}
\renewcommand{\det}{\operatorname{det}}
\newcommand{\vol}{\operatorname{vol}}

\newcommand{\rank}{\operatorname{rank}}

\newcommand{\trace}{\operatorname{trace}}
\newtheorem{lemma}[]{Lemma}
\newtheorem*{lemma*}{Old Lemma}
\newtheorem{theorem}{Theorem}
\newtheorem*{theorem*}{Old Theorem}
\newtheorem{corollary}{Corollary}

\newtheorem{definition}{Definition}

\title{A Note on  Error Bounds  for Pseudo Skeleton Approximations of Matrices}

\author{Frank de~Hoog \\
  Data 61 \\ CSIRO \\
  Black Mountain, ACT 2601 \\
  \texttt{Frank.deHoog@csiro.au} \\
  \And
  Markus Hegland \\ Mathematical Sciences Institute \\
  Australian National University \\
  Acton, ACT 2601 \\
  \texttt{Markus.Hegland@anu.edu.au}}
  
\date{}

\begin{document}
\maketitle

\begin{abstract}
  A new and improved bound for the Chebyshev norm of the error of a maximal volume pseudo-skeleton matrix approximation is presented. This bound uses all the singular values of the approximated matrix up to the $r+1$st where $r$ is the rank of the approximation. It is shown that this bound is in particular useful for the case of matrices with slowly decaying singular values where the approximation rank is of moderate size. In this case the original bounds for the maximal volume approximation are of the form $(r+1)\sigma_{r+1}$ and thus are very large and if $\sigma_{r+1}=O(1/(r+1))$ do not even indicate convergence.
  
  A new approximation based on the truncated singular value decomposition of a maximal volume approximation is analysed. It is demonstrated that this approximation is in some cases asymptotically as a good as the truncated singular value approximation of the original matrix. This new approximation is shown to have similar approximation properties as the maximal $r$-volume method analysed by Osinsky and Zamarashkin [Linear Algebra and its Applications 537 (2018), 221--249] for the case of symmetric positive definite matrices.
\end{abstract}


\section{Introduction}

A large matrix $M\in\R^{m\times n}$ can be well approximated by a matrix of rank $r$ if the singular values $\sigma_k$ of $M$ decay rapidly or if $M$ has a low effective dimension~\cite{RoyV07}. A natural approximation $M_r$ uses the \emph{truncated singular value decomposition of $M$} which is obtained from the singular value decomposition of $M$ by replacing the nonzero singular values $\sigma_{r+1},\ldots,\sigma_{\rank(M)}$ of $M$ by zero. It is assumed that the singular values are ordered by decreasing size. A challenge of using this approximation in practice are the high computational costs and the difficulty in interpreting the spaces spanned by the first $r$ singular vectors in terms of the columns and rows of $M$, respectively.

The approximation $M_r$ has be shown to be optimal in the Schatten norms (which includes the Frobenius and spectral norms, see Eckart-Young-Mirski theorem~\cite{Dax10}). We will denote the spectral norm of a matrix $M$ by $\lVert M \rVert_2$ and the error bound is an equality 
\[ \lVert M_r - M \rVert_2 = \sigma_{r+1}.\]
Here we are interested in how well the matrix elements are approximated and thus use the Chebyshev norm defined by
\[\|M\|_C = \max_{i,j} |M_{i,j}|.\]
It follows from this definition of the Chebyshev norm and the characterisation of the spectral norm as the maximal value of $(u,Mv)$ over all $u,v$ with $\lVert u \rVert = \lVert v \rVert = 1$ that $\lVert M \rVert_C \leq \lVert M \rVert_2$ and thus
\[\lVert M_r - M\rVert_C \leq \sigma_{r+1}.\]
This bound is tight and it is an equality for diagonal matrices. The simple example
\[M = \frac{1}{2} \begin{bmatrix}1&1\\1&-1\end{bmatrix}
    \begin{bmatrix}\sigma_1&0\\0&\sigma_2\end{bmatrix}
    \begin{bmatrix}1&1\\1&-1\end{bmatrix}\]
satisfies $\lVert M_1 - M\rVert_C = \sigma_2/2$ so the bound using the spectral norm $\lVert M_1 - M\rVert_2 = \sigma_2$ substantially overestimates the Chebyshev norm of the error. 


The \emph{pseudo skeleton or CUR approximation} defined below often leads methods which are computationally efficient and accurate.
\begin{definition}[Pseudo skeleton approximation]\label{def:PSA}
Let $M\in\R^{m,n}$ be a real matrix satisfying
  $$P_1 M P_2^T = \begin{bmatrix} A & B \\ C & D \end{bmatrix}$$
for some permutation matrices $P_1\in\R^{m,m}$ and $P_2\in\R^{n,n}$ and $A\in\R^{p,q}$. Furthermore, let $A^+$ be the generalised (Moore-Penrose) inverse of $A$. Then
\begin{itemize}
    \item the matrix 
  $$\widehat{M} = P_1^T \begin{bmatrix} A \\ C\end{bmatrix} A^+ \begin{bmatrix} A & B\end{bmatrix}P_2$$
is called a \emph{pseudo skeleton (or CUR) approximation of $M$} and
\item the resulting approximation is called a \emph{maximal volume approximation} if the permutations $P_1$ and $P_2$ are chosen such that 
  \begin{itemize}
      \item $\rank(A)=r$ is maximal
      \item the $r$-volume $\vol_r(A)=\sigma_1\cdots \sigma_r$ is maximal.
  \end{itemize}
\end{itemize}
\end{definition}

We will assume $p\leq q$ in the following sections and let $\widehat{M}$ always be the maximal volume pseudo skeleton approximation of $M\in\R^{n,m}$ with parameters $p, q$. Furthermore we will assume that $M$ has been reordered such that 
\[M = \begin{bmatrix}A & B \\ C & D\end{bmatrix}\] where $A$ has both maximal rank $r$ and maximal $r$-volume. Thus the permutations $P_i$ in Definition~\ref{def:PSA} are identities. 

One can show that \[r \leq p \leq q\]
and \[\rank(\widehat{M})=r.\]
It follows directly from the Definition~\ref{def:PSA} that if $r<p$ then $\rank(M)=r$ and one can show that in this case 
\[\lVert \widehat{M}-M \rVert_C = 0.\]
When we are thus proving error bounds for maximal volume skeleton approximations we only need to consider the case $r=p$. We will, however, formulate the bounds in terms of $r=\rank(A)$ as this is equal to the rank of the approximation.

The following theorem has been shown by Goreinov and Tyrtyshnikov in~\cite{GorT01}.
\begin{theorem}[Goreinov-Tyrtyshnikov\cite{GorT01}]\label{thm1}
Let $\widehat{M}$ be a maximal volume pseudo skeleton approximation of $M$ parameters $p=q$ and let $r=\rank{(\widehat{M})}$. Then
\[
      \|\widehat{M}-M\|_C \le \left(r+1\right) \sigma_{r+1} \left(M\right).
\]
\end {theorem}

 This bound suggests that the error of any CUR approximation may be substantially larger than error of the approximation based on the truncated singular value decomposition. An approximation proposed by Osinsky and Zamarashkin in~\cite{OsiZ18} gives substantially better errors in many cases. 
They propose a larger $A$ and replace the generalised inverse $A^+$ in the pseudo-skeleton approximation by the generalised inverse $A_r^+$ where $A_r$ is the rank $r$ approximation of the matrix $A$ based on the truncated singular value decomposition.  We denote this \emph{rank $r$ pseudo-skeleton approximation} by $\widetilde{M}_r$. Instead of the maximal volume of $A$ one then chooses the permutations to maximise the volume of $A_r$ which is also called the \emph{$r$-projective volume of $A$}. The error of the resulting approximation $\widetilde{M}_r$ has then been shown to be bounded by:
\begin{theorem}[Osinsky-Zamarashkin\cite{OsiZ18}]\label{thm2}
Let $\widetilde{M}_r$ be the rank $r$ pseudo skeleton approximation with maximal $r$-projective volume where the parameters $p,q$ are defined above and larger or equal to $r$.
Then 
\begin{equation} 
\left\| M - \widetilde{M}_r \right\|_C\ \le \sqrt {\left( 1 + \frac{r}{p - r + 1} \right) \left( 1 + \frac{r}{q - r + 1} \right)}\; \sigma _{r + 1}( M )
\end{equation}
\end{theorem}

One can see that if $p$ and $q$ are chosen sufficiently large, the error bound obtained in Theorem~\ref{thm2} for the rank $r$ pseudo skeleton approximation by Osinsky and Zamarashkin is close to the singular value $\sigma_{r+1}(M)$. For the case $p=q=2r-1$ the bound obtained is twice  $\sigma_{r+1}(M)$.

In many cases, however, the upper bound given in Theorem~\ref{thm1} substantially overestimates the actual error of the original maximal volume pseudo-skeleton approximation. In fact, it can be seen that the error depends on the spectral decay. Better upper bounds for the pseudo skeleton approximation have been obtained in \cite{HegdH21} in the context of nuclear norms. In practice this does not matter so much for fast decaying singular values where one typically utilises small ranks $r$ and the factor $r+1$ does not have a major impact on the error. This is different for slowly decaying singular values. For example, if $\sigma_k=1/k$ the error bound of Theorem~\ref{thm1} would suggest that the pseudo-skeleton approximation does not converge as $r \rightarrow \infty$ at all. Here, it is essential to have better error bounds. 
In the following sections we first prove better error bounds. For this, we improve the error bounds given 
in Theorem~\ref{thm1} by including a component which takes into account the rate of decay of the singular values. The bound is then applied to the large and practically important class of matrices (relating to smooth random fields and Green's functions) which have singular values of the form $\sigma_k = O(k^{-s})$. We also provide a new algorithm which has very similar properties to the maximal $r$-volume algorithms but which is based on the standard maximal volume approach. We prove
that this approach has similar properties to the algorithm from~\cite{OsiZ18}for symmetric positive definite matrices. However, we were not able to prove the same bounds for the case of non-symmetric positive matrices. The proof in the symmetric case relies substantially on Weyl's lemma states 
that the eigenvalues of a principal submatrix of a symmetric matrix $M$ is bounded by corresponding eigenvalues of the matrix
$M$. An extension of Weyl's lemma to singular values does not in general lead to useful bounds
of the singular values of submatrices and the results for the symmetric positive definite case cannot be easily generalised. 

Note that an error bound for Schatten-1 norm and symmetric positive symmetric matrices has been given in~\cite{HegdH21}. It can be seen that the Schatten-1 norm is well-defined for matrices with singular values $\sigma_k = O(k^{-s})$ where $s>1$. (Note that it is also defined for other classes of matrices but the class mentioned is the one of interest here.) In~\cite{HegdH21}, a totally different approach is used to obtain the error bounds which improve over the standard bounds of the form $(r+1)\sigma_{r+1}$.

\section{The Positive Definite Case}

In this section we improve the bound provided by Theorem~\ref{thm1} for the case where the matrix $M$ is symmetric positive definite. Let $M$ be of the form \[M = \begin{bmatrix}A & B \\ B^T & D\end{bmatrix}.\]
As $M$ is positive definite, so is $A$. Consequently, the rank of $A\in \R^{p,p}$ is $r=p$. We assume that (after a suitable symmetric permutation of rows and columns) $A$ is a principle submatrix of $M$ with maximal volume. 
The pseudo-skeleton approximation takes the form
\[\widehat{M} = \begin{bmatrix} A \\ B^T\end{bmatrix} A^{-1}
   \begin{bmatrix} A & B \end{bmatrix}.\]
   
In the following we utilise the harmonic mean of the eigenvalues $\lambda_1,\ldots,\lambda_{r+1}$ defined as
\[H(\lambda_1,\ldots,\lambda_{r+1}) = \frac{r+1}{\lambda_1^{-1}+ 
    \cdots + \lambda_{r+1}^{-1}}.\]
We then introduce parameters $\zeta_r(M)$ for $r=1,\ldots,n$ which characterise the decay of the singular values of $M$. We define them as \[\zeta_r(M) := \frac{H(\lambda_1,\ldots,\lambda_{r+1})}{\lambda_{r+1}}.
\]
One can see that $1\leq \zeta_r(M)\leq r+1$. Values of $\zeta_r(M)$ close to $r+1$ correspond to fast decay and values close to one to slow decay of the eigenvalues of $M$. 

\begin{theorem}\label{thm3}
Let $M\in\R^{n,n}$ be a symmetric positive definite matrix and let $\widehat{M}$ be a maximal volume pseudo skeleton  approximation of $M$ with parameters $p=q$. Then $r=\rank(\widehat{M})=p$ and
\[
    \left\|M-\widehat{M}\right\|_C \le \zeta_r(M)\, \lambda_{r+1}.
\]
\end {theorem}
\begin{proof}
Recall that $A$ is positive definite with rank $r=p$.

We first consider the case $n=p+1$ such that
\[M= \begin{bmatrix}A & b \\ b^T & d\end{bmatrix}\]
for some $b\in\R^n$ and $d\in\R_+$. Consequently, 
\[\widehat{M} = \begin{bmatrix} A \\ b^T\end{bmatrix} A^{-1}
   \begin{bmatrix} A & b \end{bmatrix}\]
  By direct computation one obtains
    $$\|M-\widehat{M}\|_C = d - b^T A^{-1}b.$$
  Furthermore, Schur's formula yields
     $$d - b^T A^{-1}b = \frac{\det M}{\det A}.$$
  As $M$ has full rank $r+1$, it's inverse exists and has diagonal elements
  $$M^{-1}_{i,i} = \frac{(\adj M)_{i,i}}{\det M}$$
  (by Cramer's rule) where the $i$-th diagonal elements of the adjugate $\adj M$ are the determinants of the principle $r$ by $r$ sub-matrices of $M$ obtained by omitting row $i$ and column $i$ from $M$. As $A$ is the principle sub-matrix with maximal determinant one has
    $$(\adj M)_{i,i} \leq (\adj M)_{r+1,r+1} = \det A$$
  and consequently
    $$\trace{(\adj M)} \leq (r+1) \det A.$$
  As $M^{-1}= (\det M)^{-1} \adj M$ (by Cramer's rule) one then gets
  $$d-b^TA^{-1}b = \frac{\det M}{\det A} \leq \frac{r+1}{\trace M^{-1}}$$
  from which the claimed bound follows as $M$ is invertible.
  
  The general case is then obtained as follows.
    Let the matrices $B$ and $D$ given by the block partitioning of the matrix $M$:
  $$M =\begin{bmatrix} A & B \\ B^T & D \end{bmatrix}.$$
  
  Then the Chebyshev norm of the error of the rank $r$ pseudo skeleton approximation can be seen to be
  the Chebyshev norm of the Schur complement of $A$. As $M$ is positive definite, so is the Schur complement. Thus the maximal elements of the Schur complement are positive and diagonal and so
  $$\lVert M - \widehat{M}\rVert_C = \max_{i=1,\ldots,n-r} (d_{i,i} - b_i^T A^{-1} b_i)$$
  where the $b_i$ are the $i$-th columns of $B$.
  
  Now let $k=\operatorname{argmax}_{i=1,\ldots,n-r} (d_{i,i} - b_i^T A^{-1} b_i)$ and 
  \begin{equation}\label{eqnYD}
  Y = \begin{bmatrix} A & b_k\\ b_k^T & d_{k,k}\end{bmatrix}\in\R^{r+1,r+1}
  \end{equation}
  As $Y$ is an $r+1$ by $r+1$ principle sub-matrix of $M$, the matrix $Y$ is positive definite. Furthermore, $A$ is also a maximal volume principal sub-matrix of $Y$ and hence 
  $$\widehat{Y}= \begin{bmatrix} A \\ b_k^T\end{bmatrix} A^{-1} \begin{bmatrix} A & b_k\end{bmatrix}$$ 
  is a maximal volume pseudo skeleton approximation of $Y$. One then has
  \begin{equation}\label{eqnYM}\lVert M - \widehat{M}\rVert_C = d_{k,k} - b_k^T A^{-1} b_k = \lVert Y - \widehat{Y}\rVert_C.\end{equation}
  
  From the case $n=p+1$ considered above it follows that
 \begin{equation}  \label{eqnYB}
      \|Y-\widehat{Y}\|_C \le \frac{r+1}{\lambda_1(Y)^{-1}+\cdots+\lambda_{r+1}(Y)^{-1}}.
\end{equation}
 As $Y$ is a principle sub-matrix  of $M$ a consequence of the interlacing theorem for eigenvalues~\cite{Que87} is
 $$\lambda_i(Y) \leq \lambda_i(M), \quad i=1,\ldots,r+1$$
 from which one directly gets
  \[
      \|Y-\widehat{Y}\|_C \le \frac{r+1}{\lambda_1(M)^{-1}+\cdots+\lambda_{r+1}(M)^{-1}}
 \]
 and the claimed bound follows from equation~(\ref{eqnYM}).
\end{proof}

A direct consequence is that for a large class of matrices one can obtain errors which up to a constant are of the same size as the errors of the singular value based approach.
\begin{corollary}\label{corr1thm3}
Let  $M\in\R^{n,n}$ be symmetric positive definite and let $\widehat{M}$ be the maximal volume pseudo-skeleton approximation with rank at most $r$ and parameters $p=r$ and $q$ as established in Definition~(\ref{def:PSA}). 
Furthermore, let the eigenvalues $\lambda_k$ of $M$ satisfy
\[ C_1\,k^{-s} \leq \lambda_k \leq C_2\, k^{-s}\]
for some positive real numbers $C_1,C_2$ and $s>0$ and $k=1,2,3,\ldots$. Then
\begin{equation}
  \left\|\widehat{M}-M\right\|_C \le c\, \lambda_{r+1}
\end{equation}
with $c=(s+1)C_2/C_1$.
\end{corollary}
\begin{proof}
 From the assumptions on the decay of the eigenvalues one gets
 \begin{align*}
    \zeta_r(M) &= \frac{H(\lambda_1,\ldots,\lambda_{r+1})}{\lambda_{r+1}}\\
 &
 \leq \frac{C_2}{C_1}\frac{H(1,2^{-s},\ldots,(r+1)^{-s})}{(r+1)^{-s}} \\
 &= \frac{C_2}{C_1}\frac{(r+1)^{s+1}}{(1+2^{s}+\cdots+(r+1)^{s})}\\
 &\leq \frac{C_2}{C_1}\frac{(r+1)^{s+1}}{(r+1)^{s+1}/(s+1)} = \frac{C_2}{C_1}(s+1).
 \end{align*}
 Here the last inequality is obtained by bounding the sum in the denominator by an integral. The claimed result follows.
\end{proof}

We now propose a rank $r$ approximation which is based on the maximal volume approximation with rank $r < p$ which is obtained using a truncated singular value decomposition. Corollary~\ref{cor3} provides the details together with a proof of an error bound. Note that in the following the parameter $r$ is not the rank of the maximal volume pseudo skeleton approximation but is chosen independently satisfying $r<p$ and $\lambda_{r+1}>0$.
\begin{corollary}\label{cor3}
    Let $M\in\R^{n,n}$ be a symmetric positive definite matrix and let $\widehat{M}$ be a maximal volume pseudo skeleton  approximation of $M$ with parameters $p=q$. Furthermore, let $\widehat{M}_r:=\left(\widehat{M}\right)_r$ be the approximation based on the $r$-truncated singular value decomposition of $\widehat{M}$ and $r<p$.
 Then
  \[\lVert\widehat{M}_r-M\rVert_C \leq 
    \left(\zeta_p(M)\frac{\lambda_{p+1}}{\lambda_{r+1}}+1\right) \lambda_{r+1}
    \]
    and $\rank(\widehat{M})=r$.
\end{corollary}
\begin{proof}
First, a direct application of Theorem~\ref{thm3} gives
\[\lVert \widehat{M}-M\rVert_C \leq \zeta_p(M) \frac{\lambda_{p+1}}{\lambda_{r+1}}\, \lambda_{r+1}.\]
Then, we show
\[\lVert\widehat{M}_r-\widehat{M}\rVert_C \leq \lambda_{r+1}(\widehat{M}_p) \leq \lambda_{r+1}.\]
The first inequality is the standard bound for the singular value based approximation so that it remains to show that
\[\lambda_{r+1}(\widehat{M}) \leq \lambda_{r+1}(M).\]

As
\[M-\widehat{M} = \begin{bmatrix} 0 & 0 \\ 0 & D- B^T A^{-1} B\end{bmatrix}\]
is positive semi-definite as $D- B^T A^{-1} B$ is a Schur complement and $M = \widehat{M} + (M-\widehat{M})$ one obtains
$\lambda_{r+1}(\widehat{M}) \leq \lambda_{r+1}(M)$ from Weyl's inequality.
\end{proof}

Using the bound $\zeta_p(M) \leq p+1$ one gets
\[\lVert\widehat{M}_r-M\rVert_C \leq 
    \left((p+1)\frac{\lambda_{p+1}}{\lambda_{r+1}}+1\right) \lambda_{r+1}.\]
For geometric decay of the eigenvalues, i.e.,
\[C_1 q^k \leq \lambda_k \leq C_2q^k\]
for some $q\in (0,1)$ and $C_s > 0$ one gets for a choice of $p=2r$ the bound
\[\lVert\widehat{M}_r-M\rVert_C \leq 
    \left(\frac{C_2}{C_1}(2r+1)q^r+1\right) \lambda_{r+1}.\]
It follows that the error bound converges to $\lambda_{r+1}$ for $r \rightarrow \infty$.

For a slower decay of the eigenvalues of the form
\[ C_1k^{-s} \leq \lambda_k \leq C_2k^{-s} \]
and a choice $p = \alpha(r+1)- 1$ for some $\alpha>1$ one obtains from Corollary~\ref{cor3} 
\[\lVert \widehat{M}_p-M\rVert_C \leq c \lambda_{r+1}\]
with
\[c = (C_2/C_1)^2 (s+1)\alpha^{-s}+1.\]
It follows that the suggested method has the same error order (in $r$) as the one based on the truncated singular value decomposition.

The new error bound is substantially better for matrices with eigenvalues which display moderate decay. In Figure~\ref{fig1} we provide the plots of the bounds for the case where $\lambda_k(M)=k^{-2}$. 
\begin{figure}
    \begin{center}
        \input{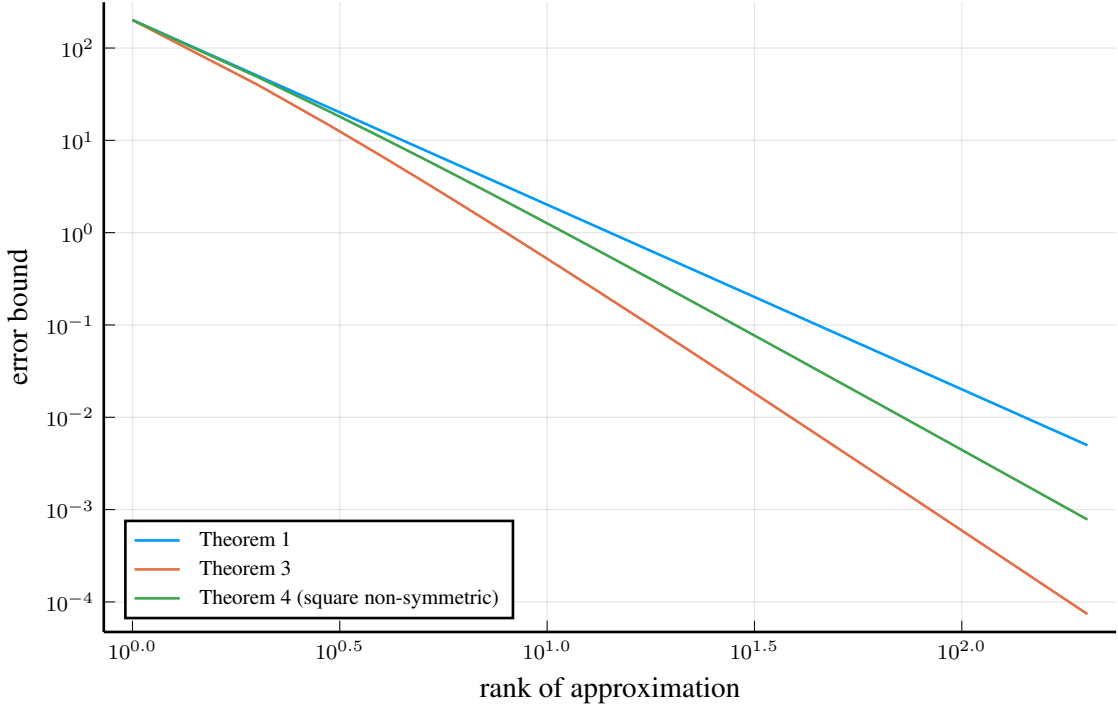}
    \end{center}
    \caption{\label{fig1}Comparison of error bounds for the case $\sigma_k(M)=k^{-2}$}
\end{figure}
Furthermore, the error bound is exact for the limiting case of $M=I$ (where $I$ is the identity matrix). It follows directly from Theorem~\ref{thm3} that any pseudo-skeleton approximation of the identity is an optimal approximation in the Chebyshev norm. While this can be also shown directly, it does not, however, follow from the theorems mentioned in the introduction.

\newpage 

\section{The General Case}

The analysis of the approximation error for non-symmetric matrices $M\in\R^{m,n}$ is based on the singular values of $M$. Here we assume that
\[M = \begin{bmatrix} A & B \\ C & D \end{bmatrix}\]
where  $A\in \R^{p,q}$ is the $p\times q$ submatrix with maximal rank which has maximal volume. We assume that the parameters established in Definition~\ref{def:PSA} satisfy $p\leq q$. As we are mostly interested in error bounds we can assume that any given matrix has been permuted such that it has this property. Furthermore, the case $p > q$ is treated by considering the transpose. 

A \emph{pseudo-skeleton approximation} is then given by 
\[\widehat{M} = \begin{bmatrix} A \\ C\end{bmatrix} A^+ 
\begin{bmatrix} A & B\end{bmatrix}.\]
In the following we will use the function $\zeta_r$ introduced in the previous section. In particular, one has 
\[\zeta_r(M^TM) = \frac{H(\sigma_1^2,\ldots,\sigma_{r+1}^2)}{\sigma_{r+1}^2}.
\]

\begin{theorem}\label{thm4}
Let the matrix $\widehat{M}$ with parameters $p\leq q$ be a maximal volume  pseudo-skeleton approximation of a matrix $M\in\R^{m,n}$ with $\rank(M)\geq p+1$. Then 
\begin{equation}
  \left\|\widehat{M}-M\right\|_C \le
  \sqrt{\frac{\zeta_r(M^TM)}{1-r/(q+1)}}\, \sigma_{r+1}(M).   
\end{equation}
\end {theorem}

The following lemma is essential for the proof of the main theorem of this section.
\begin{lemma}\label{lemma3}
Let $p\leq q$ be two integers and  
\[W=\begin{bmatrix}A & b\end{bmatrix}P\in\R^{p,q+1}\] where $P$ is a permutation matrix such that $\vol(A) \geq \vol(W_i)$ for all submatrices $W_i\in\R^{p,q}$ of $W$ obtained by omitting the $i$-th column. Then
\[\vol(W) \leq \frac{1}{\sqrt{1-p/(q+1)}} \vol(A).\]
\end{lemma}
\begin{proof}
The Lemma is a direct consequence of the Cauchy-Binet theorem, see  
Mikhalev and Oseledets~\cite[Lemma 4.1]{MikO18}. Specifically, one gets
\[\det(WW^T) = \frac{1}{q+1-p} \sum_{i=1}^{q+1} \det(W_iW_i^T) \leq
  \frac{q+1}{q+1-p} \det(AA^T)\]
and using the definition of the volume one obtains the claimed bound.
\end{proof}

We now proceed with the proof of the theorem. First we prove the case $m=p+1$ $n=q+1$ and then we use this result to prove the bound for general $m>p$ and $n>q$.

\begin{proof}[Proof of Theorem~\ref{thm4}]
Consider the case $m=p+1$ and $n=q+1$. In this case it follows that $M$ has full rank $p+1=r+1$. As in this case
\[M = \begin{bmatrix}A & b \\ c^T & d \end{bmatrix}\]
for some vectors $b\in\R^p$, $c\in\R^q$ and a positive real number $d$, 
the pseudo skeleton approximation $\widehat{M}$ satisfies
\[ \widehat{M} = \begin{bmatrix}A & b \\ c^TA^+A & c^TA^+b \end{bmatrix}.\]
It follows that the Chebyshev norm of the error is
\begin{equation}\lVert \widehat{M} - M\rVert_C = \max \left(\lVert u \rVert_C,\,\lvert \gamma \rvert\right).\label{eqn14}\end{equation}
where $u=(I-AA^+)\,c$ and $\gamma=d-c^TA^+b$.

 From the definition of the adjugate of a matrix, we have
\[\det\left( M M^T\right) \left(MM^T\right)^{-1} = \adj \left(MM^T\right)\]
and taking the trace of each side, we obtain
\[\vol\left( M \right)^2 \left\|M^+\right\|_F^2
=\trace{\left( \adj \left(MM^T\right) \right)}.\]

It can be seen that the right hand side consists of the sum of the squares of the volumes of all submatrices of size $p\times(q+1)$ if $M$. As there are $r+1$ such submatrices, an application of Lemma~\ref{lemma3} gives
$$\trace{\left( \adj \left(MM^T\right) \right)} \leq \frac{(r+1)(q+1)}{q+1-r} \vol{A}^2.$$
Thus
\begin{equation}\vol\left( M \right)^2 \left\|M^+\right\|_F^2 \leq \frac{(r+1)(q+1)}{q+1-r} \vol{A}^2.\label{eqn15}\end{equation}

Then one gets
$$\begin{bmatrix}I & 0 \\ -c^TA^+ & 1\end{bmatrix} M =
  \begin{bmatrix} A & b \\ u^T & \gamma\end{bmatrix}$$
and, as the determinant of the first factor is one 
$$\vol(M) = \vol\left(\begin{bmatrix}A & b \\ u^T & \gamma\end{bmatrix}\right)$$
and hence
$$\vol(M)^2 = \det(MM^T)=\det\left(\begin{bmatrix}AA^T + bb^T & \gamma b \\ \gamma b^T & \lVert u \rVert_2^2+\gamma^2\end{bmatrix}\right).$$
Using the Schur complement of the lower right element of the block $2\times 2$ matrix gives
\begin{align*}
    \vol(M)^2 &= (\lVert u \rVert_2^2 + \gamma^2)) 
       \det\left(AA^T + \frac{\lVert u \rVert_2^2}{\lVert u \rVert_2^2+\gamma^2} bb^T\right).
\end{align*}
It follows that
\[ \vol(M)^2 \geq (\lVert u \rVert_2^2 + \gamma^2) \vol(A)^2.\]

Combining this inequality with equation~\ref{eqn15} gives 
\begin{align*}
    \lVert \widehat{M}-M\rVert_c^2 &= \max(\lVert u\rVert_C^2,\vert \gamma\rvert^2) \leq \lVert u\rVert_2^2 + \gamma^2 \\ &\leq \frac{\vol(M)^2}{\vol(A)^2} \leq \frac{(r+1)(q+1)}{q+1-r}\, \lVert M^+\rVert_F^{-2}.
\end{align*}
Introducing $\zeta_r$ provides the claimed bound.

We now consider the case of general $n$ and $m$.
  The error matrix of the pseudo skeleton approximation is
  $$\widehat{M}-M = -\begin{bmatrix} 0 & 0 \\ C(I-AA^+) & D - CA^+B\end{bmatrix}.$$
  Let $r+i$ be the index of the row of the error which contains the element with largest absolute value. Furthermore, let $j$ be the column index of of the column of $D-CA^+B$ which contains the largest absolute value. The Chebyshev norm of the error matrix is then
  $$\lVert \widehat{M}-M \rVert_C =
    \max \left\{\lVert c_i^T(I-AA^+)\rVert_C,\, \lvert d_{i,j}-c_i^T A^+ b_j\rvert\right\}$$
  where $c_i^T$ and $b_j$ are the $i$-th row of $C$ and $j$-th column of $B$, respectively.
  
  One can see that the right-hand side of the above error formula is just the Chebyshev norm of the error of the maximal volume rank $r$ pseudo skeleton approximation of the matrix 
  \begin{equation}  \label{eqnZD} 
  Z:=\begin{bmatrix} A & b_j \\ c_i^T & d_{i,j} \end{bmatrix},
  \end{equation}
  Now we apply the bound obtained for the special case of $m=p+1$ and $n=q+1$ to get
  \begin{equation} \label{eqnEB}
  \left\|\widehat{M}-M\right\|_C \le
   \sqrt{\frac{(r+1)(q+1)/(q+1-r)}{\sigma_1^{-2}(Z)+\cdots+\sigma_{r+1}^{-2}(Z)}}.
   \end{equation}
   By the interlacing theorem for singular values~\cite{Que87} one has $$\sigma_k(Z) \leq \sigma_k(M), \quad k=1,\ldots,r+1$$
   as $Z$ is a submatrix of $M$. The claimed bound then follows directly.
\end{proof}



A direct consequence is that for a large class of matrices one can obtain errors which up to a constant are of the same size as the errors of the singular value based approach.
\begin{corollary}\label{corr2thm4}
Let the matrix $M\in\R^{m,n}$ with $\rank M \geq r+1$ and let $\widehat{M}$ be the maximal volume pseudo-skeleton approximation with rank at most $r$ and parameters $p=r$ and $q$ as established in Definition~(\ref{def:PSA}). 
Furthermore, let $q=\alpha r-1$ for some rational $\alpha>1$ and let the singular values $\sigma_k$ of $M$ satisfy
\[ \frac{C_1}{k^{2s}} \leq \sigma_k^2 \leq \frac{C_2}{k^{2s}}\]
for some $s>0$ and $k=1,2,3,\ldots$. Then
\begin{equation}
  \left\|\widehat{M}-M\right\|_C^2 \le
  \frac{\alpha}{\alpha-1} \frac{C_2}{C_1} (2s+1)
  \, \sigma_{r+1}(M)^2.   
\end{equation}
\end{corollary}
\begin{proof}
 The error bound of Theorem~\ref{thm4} has the three factors
 $1/\sqrt{1-r/(q+1)}$, $\sqrt{\zeta_r(M^TM)}$ and $\sigma_{r+1}(M)$. 
 The choice of $q$ ensures that the first factor is
 \[1/\sqrt{1-r/(q+1)} = 1/\sqrt{1-1/\alpha}.\]
 Taking the square gives the first factor of the bound of the corollary.
 
 Then we have for the squared second factor
 \begin{align*}
    \zeta_r(M^TM) &= \frac{H(\sigma_1^2,\ldots,\sigma_{r+1}^2)}{\sigma_{r+1}^2}\\
 &
 \leq \frac{C_2}{C_1}\frac{H(1,2^{-2s},\ldots,(r+1)^{-2s})}{(r+1)^{-2s}} \\
 &= \frac{C_2}{C_1}\frac{(r+1)^{2s+1}}{(1+2^{2s}+\cdots,+(r+1)^{2s})}\\
 &\leq \frac{C_2}{C_1}\frac{(r+1)^{2s+1}}{(r+1)^{2s+1}/(2s+1)} = \frac{C_2}{C_1}(2s+1).
 \end{align*}
 Here the last inequality is obtained by bounding the sum in the denominator by an integral. The claimed result follows.
\end{proof}

Using the same approach as in as Corollary~\ref{cor3} one gets
  \[\lVert\widehat{M}_r-M\rVert_C \leq 
    \left(\sqrt{\zeta_p(M^TM)}\frac{\sigma_{p+1}}{\sigma_{r+1}}+
    \frac{\sigma_{r+1}(\widehat{M})}{\sigma_{r+1}}\right) \sigma_{r+1}.
    \]
An application of Weyl's lemma provides
$\sigma_{r+1}(\widehat{M}) \leq  \sigma_{r+1}(M)+ \lVert \widehat{M}-M\rVert_2.$
The problem is that for large matrices the spectral norm can be very large compared to the Chebyshev norm, see~\cite{OsiZ18} and in particular Theorem~2 in that paper for more details so that Corollary~\ref{cor3} cannot be easily generalised to the non-symmetric case short of assuming a condition like $\sigma_{r+1}(\widehat{M}) \leq C\sigma_{r+1}$ for some $C>1$.

\section{Concluding Remarks}

Chebyshev matrix approximation is very useful as it provides error bounds for the matrix elements. Unlike in the case of $\ell_2$ approximation, little is known about the optimal errors (Chebyshev singular values) and there is no known exact algorithm. The commonly used error bounds for maximal volume pseudo skeleton approximations are both tight and practically useful for fast decaying singular values. However, they are less accurate and not useful in the case where the singular values decay slowly, e.g. where $\sigma_k = O(k^{-2})$ for small $s$. Continuing on our work in~\cite{HegdH21} we have presented new improved error bounds and in particular demonstrated their improvement for matrices with slowly decaying singular values. 

In~\cite{GorT11} the authors have presented a new algorithm which combines a pseudo skeleton approximation with truncated singular value decomposition of this approximation. As noted previously, they show that this new method provides approximations which are up to a constant as good as the one based on the truncated singular value decomposition. The method is based on computing a pseudo skeleton approximation with maximal $r$-volume. Here we present a similar approach but use a traditional maximal volume pseudo skeleton approximation instead. We provide the new improved error bounds for this method which, at least in the symmetric positive definite case, has similar properties to those one proposed in˜\cite{GorT11}. We also provide bounds of this method for the non-positive definite case but they do depend on the singular values of the pseudo skeleton approximations. Further research may further clarify this case.

The bounds discussed so far have been all in terms of the singular values which are the bounds of best approximations using Schatten norms. A next step would be bounds based on best Chebyshev approximations. A good first collection of algorithms and results for pseudo-skeleton matrix approximations has been published in~\cite{ZMT22} during the revision of this manuscript. An extension of the research in pseudo-skeleton approximation has considered the approximation of tensors, see~\cite{Gor08}.

\printbibliography

\end{document}